\documentclass{amsart}
\usepackage{mathrsfs}
\usepackage{amsfonts}
\usepackage{amsfonts,amssymb,amsmath,amsthm}
\usepackage{url}
\usepackage{enumerate}

\theoremstyle{plain}
\newtheorem{theorem}{Theorem}[section]
\newtheorem{thm}{Theorem}[section]
\newtheorem{lemma}[thm]{Lemma}
\newtheorem{corollary}[thm]{Corollary}
\newtheorem{corollary*}[thm]{Corollary*}
\newtheorem{proposition}[thm]{Proposition}
\newtheorem{proposition*}[thm]{Proposition*}
\newtheorem{conjecture}{Conjecture}

\newtheorem{open problem}[thm]{Open Problem}
\newtheorem{remark}[thm]{Remark}
\theoremstyle{definition}

\newtheorem{defn-thm}[thm]{Definition-Theorem}
\numberwithin{equation}{section}
\newenvironment{enumerateroman}{\begin{enumerate}[\textup{(}i\textup{)}] }{\end{enumerate}}

\title[A sharp convergence theorem for the mean curvature flow ]
{ A sharp convergence theorem for the mean curvature flow in spheres I }

\author{Dong Pu}
\address{Faculty of Science, Tongji Zhejiang College,
             Jiaxing,  314051, People's Republic of China}
\email{pudong@zju.edu.cn}

\thanks{ }



\begin{document}


\begin{abstract}  In this paper, we prove a sharp convergence theorem for the mean curvature flow of arbitrary codimension in spheres which improves Baker's convergence theorem. In particular, we obtain a new
differentiable sphere theorem for submanifolds in spheres.
\end{abstract}
 \maketitle

\section{Introduction}
Let $M$ be an $n$-dimensional compact submanifold isometrically
immersed in a Riemannian manifold $N^{n+p}$. Denoted by $F_0$ the
isometric immersion. We consider the one-parameter family
$F_{t}=F(\cdot,t)$ of immersions $F_{t}:M\rightarrow N^{n+p}$ with
corresponding images $M_{t}=F_{t}(M)$ which satisfies
\begin{equation}
  \left\{ \begin{array}{l}
\frac{\partial}{\partial t}F(x,t)=H(x,t),\\
F(x,0)=F_0(x),
\end{array} \right.
\end{equation}
where $H(x,t)$ is the mean curvature vector of $M_t$. We call
$F_{t}:M\rightarrow N^{n+p}$ the mean curvature flow with initial
value $F_0$. In 1980's,  Huisken proved the convergence theorems for
the mean curvature flow of hypersurfaces under certain conditions in
a series of papers.
 For the initial hypersurfaces satisfying the convexity condition, Huisken \cite{H1,H2} proved that the
solution of the mean curvature flow converges to a point as the time
approaches the finite maximal time. Motivated by the rigidity theorem for hypersurfaces with constant mean curvature in spheres due to Okumura \cite{Ok}, Huisken \cite{H3}
proved the convergence theorem for the mean curvature flow of  hypersurfaces under the curvature pinching condition for
$|H|^2$ and the squared norm of the second fundamental form $|A|^2$ in the
sphere $\mathbb{S}^{n+1}$.
\begin{theorem}\label{Huisken}
Let $F_0 :M^{n} \rightarrow \mathbb{S}^{n+1}(\frac{1}{\sqrt{\bar K}})$ be a smooth compact
hypersurface in a sphere with constant curvature $\bar K$.
Assume $M$ satisfies
\begin{equation}
|A|^2<\begin{cases}
            \frac{3}{4}|H|^2+\frac{4}{3}\bar K, \ &n = 2, \\
            \frac{1}{n-1}|H|^2+2\bar K, \ &n \geq 3.
        \end{cases}
\end{equation}
Then
  the mean curvature flow with the initial value $F_0$
either converges to a round point in finite time, or
converges to a total geodesic sphere of $\mathbb{S}^{n+1}(\frac{1}{\sqrt{\bar K}})$ as
$t\rightarrow\infty$.
\end{theorem}

For the mean curvature flow of arbitrary codimension in
the Euclidean space, Andrews and Baker \cite{AB} investigated the
convergence problem. In \cite{B}, Baker studied the mean curvature flow of arbitrary codimension in spheres and obtained the following result.
\begin{theorem}\label{AB2}
Let $F_0 :M^{n} \rightarrow \mathbb{S}^{n+p}(\frac{1}{\sqrt{\bar K}})$ be a smooth compact
submanifold in a sphere with constant curvature $\bar K$.
Assume $M$ satisfies
\begin{equation}
|A|^2\leq\begin{cases}
            \frac{4}{3n}|H|^2+\frac{2(n-1)}{3}\bar K, \ &n = 2,3, \\
            \frac{1}{n-1}|H|^2+2\bar K, \ &n \geq 4.
        \end{cases}
\end{equation}
Then
  the mean curvature flow with the initial value $F_0$
either converges to a round point in finite time, or
converges to a total geodesic sphere of $\mathbb{S}^{n+p}(\frac{1}{\sqrt{\bar K}})$ as
$t\rightarrow\infty$.
\end{theorem}

For $n \geq 4$, the compact submanifold  $M^n=\mathbb{S}^1(r)\times \mathbb{S}^{n-1}(s)\subset \mathbb{S}^{n+1}(1)\subset \mathbb{S}^{n+p}(1)$ with $r^2+s^2=1$ implies that the pinching conditions of Theorem \ref{Huisken} and Theorem \ref{AB2} are sharp. In fact, the sharp pinching condition means the linear relationship of $|A|^2$, $|H|^2$ and $\bar K$ is sharp. If the relationship is not linear, the pinching condition can be improved. Afterwards, Lei-Xu \cite{LX2} also obtained a sharp convergence theorem for mean curvature flow of high codimension in spheres. Set
\begin{equation}
\alpha (x) = n \bar K+ \frac{n}{2 ( n-1 )} x -
\frac{n-2}{2
   ( n-1 )} \sqrt{x^{2} +4 ( n-1 ) \bar K x},
\end{equation} they proved the following theorem.
\begin{theorem}\label{LX2}
Let $F_0 :M^{n} \rightarrow \mathbb{S}^{n+p}(\frac{1}{\sqrt{\bar K}})$ be an n-dimensional $(n\geq6)$ smooth compact
submanifold in a sphere with constant curvature $\bar K$.
Assume $M$ satisfies
\begin{equation}
|A|^2< \gamma(n,|H|,\bar K).
\end{equation}
Then
  the mean curvature flow with the initial value $F_0$
either converges to a round point in finite time, or
converges to a total geodesic sphere of $\mathbb{S}^{n+p}(\frac{1}{\sqrt{\bar K}})$ as
$t\rightarrow\infty$.

Here $\gamma(n,|H|,\bar K)$ is an explicit positive scalar defined by
$$\gamma(n,|H|,\bar K)=\min\{\alpha(|H|^2), \beta(|H|^2) \},$$
where
\begin{eqnarray*}
  \beta(x) &=&  \alpha(x_0)+ \alpha'(x_0)(x-x_0)+\frac12\alpha''(x_0)(x-x_0)^2, \\
  x_0 &=& \frac{2n+2}{n-4}\sqrt{n-1}(\sqrt{n-1}-\frac{n-4}{2n+2})^2\bar K.
\end{eqnarray*}
\end{theorem}
The scalar $\gamma(n,|H|,\bar K)$ in Theorem \ref{LX2} satisfies the following: (i) $\gamma(n,|H|,\bar K)>\frac{1}{n-1}|H|^2+2\bar K$; (ii) $\gamma(n,|H|,\bar K)>\frac76\sqrt{n-1}\bar K$; (iii) $\gamma(n,|H|,\bar K)=\alpha(|H|^2)$ when $|H|^2\geq x_0$.

Lawson-Simons \cite{LS} proved the topological sphere theorem in the unite sphere under the pinching condition $| A |^2<2\sqrt{n-1}$. Applying the convergence results of Hamilton and
Brendle for Ricci flow and the Lawson-Simons formula for the
nonexistence of stable currents, Xu and Zhao \cite{MR2550209} proved
a differentiable sphere theorem in spheres.
\begin{theorem}\label{ds} Let $M$ be an $n$-dimensional ($n \geq 4$)
oriented complete submanifold in the unit sphere $\mathbb{S}^{n+q}$.
Then\\
\hspace*{2mm}$(i)$ if $n=4,5,6$ and $\sup_M(| A |^2-2\sqrt{n-1})<0$, then $M$ is diffeomorphic to $\mathbb{S}^{n}$;\\
\hspace*{2mm}$(ii)$ if $n\geq7$ and $|A|^2<2\sqrt{2}$, then $M$ is
diffeomorphic to  $\mathbb{S}^{n}$.\end{theorem}

 Motivated by the rigidity and topological sphere theorems for submanifolds in spheres, Lei-Xu proposed the conjecture that $\alpha(|H|^2)$ is the optimal pinching condition for mean curvature flow in spheres.
\begin{conjecture}
  Let $M_{0}$ be an n-dimensional complete submanifold in the sphere
  $\mathbb{S}^{n+p}(\frac{1}{\sqrt{\bar K}})$. Suppose that
  $\sup_{M_{0}} \big( | h |^{2} - \alpha (| H |^2) \big) <0$. Then the mean
  curvature flow with initial value $M_{0}$ converges to a round point in
  finite time, or converges to a totally geodesic sphere as $t \rightarrow
  \infty$. In particular, if $| A |^{2} <2 \sqrt{n-1}\bar K$, $M_{0}$ is diffeomorphic to $\mathbb{S}^{n}$.
\end{conjecture}

For mean curvature flow of submanifolds in hyperbolic spaces,  Liu-Xu-Ye-Zhao \cite{LXYZ} proved the convergence theorem.
Recently, Lei-Xu \cite{LX} proved a convergence theorem of arbitrary codimension in hyperbolic spaces that the initial submanifold $M^n$ of
dimension $n(\ge6)$ under the optimal pinching condition $
|A|^2 < \alpha(|H|^2).$
Note that initial submanifolds in the almost all convergence
results implies the positive sectional curvatures. However, Lei-Xu's convergence theorems imply
that the Ricci curvatures of the initial submanifolds are positive, but don't imply
the positivity of the sectional curvatures. Therefore, their
convergence theorems also imply the differentiable sphere theorems
for submanifolds with positive Ricci curvatures. For other results and applications for the mean curvature flow, we refer the readers to \cite{BN,H2,H4,LX4,LXYZ1,Sm}.

Motivated by above theorems, we investigate the submanifold $M^n$ $(n\geq3,p=1$ or  $n\geq4)$pinched by a sharp curvature pinching condition for the mean curvature flow of arbitrary codimension in the sphere ${S}^{n+p}(\frac{1}{\sqrt{\bar K}})$.
Putting
\begin{equation}
  b(x)=(1-\delta)\big(\frac{x}{n-1}+2 \bar K\big)+\delta \alpha(x),\ x\geq0,
\end{equation}
where $\delta=\frac{\sqrt{12n+9}-7}{2(n-2)},\ n=4,5,6.$
The pinching condition $b(|H|^2)$ is obvious sharp and satisfies $b(|H|^2)> \frac{|H|^2}{n-1}+2\bar K.$ Then we prove the following sharp convergence theorem.
\begin{thm}\label{main2}
Let $F_0 :M \rightarrow \mathbb{S}^{n+p}(\frac{1}{\sqrt{\bar K}})$ be an n-dimensional $(n\geq4)$ smooth compact submanifold immersed in the sphere. If $M$ satisfies
\begin{equation}
|A|^2\leq\begin{cases}
            \  b(|H|^2), \ &n=4,5,6, \\
            \\
            \sqrt{\big(\frac{|H|^2}{n-1}+2\bar K\big)^2+(2n-4)\bar K^2}, \ &n \geq 7.
        \end{cases}
\end{equation}
then the mean curvature flow with the initial value $F_0$  converges to a round point in finite time, or
converges to a total geodesic sphere of $\mathbb{S}^{n+p}(\frac{1}{\sqrt{\bar K}})$ as
$t\rightarrow\infty$.
\end{thm}

\begin{remark}
\rm{(i)} Consider the compact submanifold  $M^n=\mathbb{S}^1(r)\times \mathbb{S}^{n-1}(s)\subset \mathbb{S}^{n+1}(1)\subset \mathbb{S}^{n+p}(1)$ with $r^2+s^2=1$. We have $|A|^4 - \big(\frac{|H|^2}{n-1}+2\big)^2-(2n-4)=\frac{(n-1)^2-1}{(n-1)^2}\cdot\frac{s^4}{r^4}$, which implies that the pinching condition is sharp.\\
\rm{(ii)} Since $\sqrt{\big(\frac{|H|^2}{n-1}+2\bar K\big)^2+(2n-4)\bar K^2}> \frac{|H|^2}{n-1}+2\bar K$, this convergence theorem improves Theorem \ref{Huisken} and Theorem \ref{AB2}. In fact, whatever the form of the pinching condition is, the relationship of $|A|^2$ and $|H|^2$, i.e., $\frac{1}{n-1}$ is optimal. So this pinching condition largely improved the proportion of $\bar K$.\\
\rm{(iii)} This convergence theorem implies that the Ricci curvatures of the initial submanifold is positive, but doesn't imply
the positivity of the sectional curvature.\\
\rm{(iv)} For $n\geq3,p=1$, the mean curvature flow of hypersurfaces in the sphere convergence under the sharp condition  $$|A|^2\leq\sqrt{\Big(\frac{|H|^2}{n-1}+2\bar K\Big)^2+(2n-4)\bar K^2}.$$
\end{remark}

In particularity, the convergence theorem implies a sharp differentiable sphere theorem in the sphere. Noting that $\big(\frac{|H|^2}{n-1}+2\big)^2+2n-4\geq2n$, we obtain a new
differentiable sphere theorem which improves Theorem \ref{ds}.
\begin{corollary} Let $M$ be an $n$-dimensional ($n \geq 7$) smooth
compact submanifold in the unit sphere $\mathbb{S}^{n+q}$.
If $|A|^2\leq\sqrt{2n}$, then $M$ is
diffeomorphic to  $\mathbb{S}^{n}$.\end{corollary}

$  \ $
\section{Preliminaries}

Let $( M^{n} ,g )$ be the $n$-dimensional $(n\geq3,p=1$ or  $n\geq4)$ Riemannian submanifold isometrically immersed in a simply connected space form
$\mathbb{F}^{n+d} ( \bar K)$ with constant curvature $\bar K$. Denote by
$\bar{\nabla}$ the Levi-Civita connection of the ambient space
$\mathbb{F}^{n+d} $. We use the same symbol $\nabla$ to represent the
connection of the tangent bundle $T M$ and the normal bundle $N M$. Denote by
$( \cdot )^{\top}$ and $( \cdot )^{\bot}$ the projections onto $T M$ and $N
M$, respectively. For $u,v \in \Gamma ( T M )$, $\xi \in \Gamma ( N M )$, the
connection $\nabla$ is given by $\nabla_{u} v= ( \bar{\nabla}_{u} v )^{\top}$
and $\nabla_{u} \xi = ( \bar{\nabla}_{u} \xi )^{\perp}$. The second fundamental
form of $M$ is defined by
\[ A ( u,v ) = ( \bar{\nabla}_{u} v )^{\perp} . \]

Let $\{ e_{i}   \,|\,  1 \le i \le n \}$ be a local orthonormal frame
for the tangent bundle and $\{ \nu_{\alpha}   \,|\,  n+1 \le \alpha \le
n+d \}$ be a local orthonormal frame for the normal bundle. Let $\{ \omega_{i}
\}$ be the dual frame of $\{ e_{i} \}$. With the local frame, the first and
the second fundamental forms can be written as $g= \sum_{i} \omega^{i} \otimes
\omega^{i}$ and $A= \sum_{i,j, \alpha} h_{i j \alpha}   \omega^{i} \otimes
\omega^{j} \otimes \nu_{\alpha}=\sum_{i,j} h_{i j }   \omega^{i} \otimes
\omega^{j}$, respectively. The mean curvature vector is
given by
\[ H= \sum_{\alpha} H_{\alpha} \nu_{\alpha} , \hspace{1em} H_{\alpha} =
   \sum_{i} h_{i i \alpha} . \]
Let
$\mathring{A} =A - \tfrac{H}{n}  g$ be the
traceless second fundamental form, then we have
$|\mathring{A}|^{2} = | A |^{2} - \frac{|H|^{2}}{n} $ and $|\nabla \mathring A |^{2} = | \nabla A |^{2} - \frac{|\nabla H|^{2}}{n} $.
Denote by $\nabla^{2}_{i,j} T= \nabla_{i} ( \nabla_{j} T ) -
\nabla_{\nabla_{i} e_{j}} T$ the second order covariant derivative of tensors.
Then the Laplacian of a tensor is defined by $\Delta T= \sum_{i}
\nabla^{2}_{i,i} T$.

We have the following evolution equations for the mean curvature flow.
\begin{lemma}[\cite{AB,B}]\label{ee}
\begin{eqnarray}
  \nabla_{\partial_t} h_{ij} &=& \Delta h_{ij} +h_{ij}\cdot h_{pq}h_{pq}+h_{iq}\cdot h_{qp}h_{pj}+h_{jq}\cdot h_{qp}h_{pi} \label{ee1} \\
   && -2h_{iq}\cdot h_{jp}h_{pq}+2\bar KHg_{ij}-n\bar Kh_{ij}, \nonumber \\
  \nabla_{\partial_t} H &=& \Delta H +H\cdot h_{pq} h_{pq}+n\bar K H, \\
  \frac{\partial}{\partial t}|A|^2 &=& \Delta|A|^2-2|\nabla A|^2+2R_1+4\bar K|H|^2-2n\bar K |A|^2, \\
  \frac{\partial}{\partial t}|H|^2 &=& \Delta|H|^2-2|\nabla H|^2+2R_2+2n\bar K|H|^2,\\
  \frac{\partial}{\partial t}|\mathring A|^2 &=& \Delta|\mathring A|^2-2|\nabla \mathring A|^2+2R_1-\frac2n R_2-2n\bar K |\mathring A|^2.
\end{eqnarray}

where
\begin{eqnarray}
R_1&=&2\sum\limits_{\alpha,\beta}(\sum\limits_{i,j}h_{ij\alpha}h_{ij\beta})^2+|{Rm}^{\perp}|^2,\\
R_2&=&\sum\limits_{i,j}(\sum\limits_{\alpha}H_\alpha h_{ij\alpha})^2.
\end{eqnarray}
\end{lemma}

We have the following following curvature estimates.
\begin{lemma}[\cite{B}]\label{eq}
\begin{align}
|\nabla A | ^ 2 &\geq \frac  {3 }{ n+2 } | \nabla H | ^ 2,  \label{eqn_gradient1} \\
R_{2} &= | \mathring{A} |^{2} | H |^{2} + \frac{1}{n} | H |^{4} -P_{2} | H|^{2},  P_{2} = \sum_{\alpha >1} \left( \mathring{h}_{i j \alpha} \right)^{2},\\
   R_3&=\sum H^\alpha h^\alpha_{ij} h^\beta_{jk} h^\beta_{ki},\\
R_{1} - \frac{1}{n} R_{2} &\leq | \mathring{A} |^{4} + \frac{1}{n}
  | \mathring{A} |^{2} | H |^{2} +2P_{2} | \mathring{A} |^{2} - \frac{3}{2} P_{2}^{2} -
  \frac{1}{n} P_{2} | H |^{2},\\
R_3-R_1\geq&\frac{|\mathring A|^2|H|^2}{2(n-1)}-\frac n2(|\mathring A|^2-P_2)\\
&-\max\{4,\frac{n+2}{2}\}(|\mathring A|^2-P_2)P_2-\frac32P^2_2.\notag
\end{align}
\end{lemma}

For convenience, we denote $$a ( x )=\sqrt{\Big(\frac{x}{n-1}+2\bar K\Big)^2+(2n-4)\bar K^2},\ \ \ \mathring{a} ( x )=a( x )-\frac{x}{n}.$$
Moreover, we prove the following inequalities.
\begin{lemma}
  $\label{app}$For $x \geq 0$,
  $\mathring{a}$ has the following properties.
  \begin{enumerateroman}
   \item $\frac{4x (\mathring{a}')^2}{\mathring{a}}  < 1,$

    \item $2x  \mathring{a}''  + \mathring{a}'  < \frac{2
    ( n-1 )}{n ( n+2 )},$

    \item $\frac{n-2}{\sqrt{n(n-1)}}\sqrt{x\mathring{a}}+\mathring{a}  < \frac{x}{n}+n\bar K,$

    \item $n\bar K(\mathring{a}+x\mathring{a}')-a( \mathring{a}-x\mathring{a}')\geq \frac{2n(n-2)(n-1)^2\bar K^4}{\big(x+\sqrt{2+\sqrt{2n}}(n-1)\bar K\big)^2} $,

    \item $2\mathring{a}-\frac{x}{n}+x\mathring{a}'\leq 2\sqrt{2n}\bar K-\frac{(n-4)x}{n(n-1)}-\frac{6(\sqrt{2n}-2)\bar K x}{3x+2\sqrt{2n}(n-1)\bar K}$,

    \item $\frac{x}{n-1}\left( a  +n\bar K
    \right) - (\frac{x}{n-1}+2n\bar K) \left(  \mathring{a} +
    a  -n\bar K - x\mathring{a}' \right) \\
       < -\frac{2x\bar K}{n-1}+2n(n-4)\bar K^2 $.
  \end{enumerateroman}
\end{lemma}

\begin{proof}
  By direct computations, we get
  \[ \mathring{a}=\sqrt{\Big(\frac{x}{n-1}+2\bar K\Big)^2+(2n-4)\bar K^2}-\frac xn,  \]

  \[ \mathring{a}'  = \frac{\frac{x}{n-1}+2\bar K}{(n-1)\sqrt{\Big(\frac{x}{n-1}+2\bar K\Big)^2+(2n-4)\bar K^2}}-\frac1n<\frac{1}{n(n-1)} ,  \]

  \[\mathring{\alpha}''  = \frac{2( n-2 ) \bar K^{2}}{(n-1)^2 \left(\sqrt{\Big(\frac{x}{n-1}+2\bar K\Big)^2+(2n-4)\bar K^2}\right)^3} .\]

(i)
\begin{equation*}
  \frac{4x (\mathring{a}')^2}{\mathring{a}}  < \frac{x }{n(n-1)\mathring{a}}<1.
\end{equation*}

 (ii)
 \begin{align*}
   2x  \mathring{a}''  <& \frac{4( n-2 )x \bar K^{2}}{(n-1)^2 \Big(\frac{x}{n-1}+2\bar K\Big)\left(\Big(\frac{x}{n-1}+2\bar K\Big)^2+(2n-4)\bar K^2\right)} \\
     < &\frac{4( n-2 ) \bar K^{2}}{\frac{x^2}{n-1}+ 6\bar Kx+\frac{4n(n-1)^2}{x}\bar K^3+  (2n+8)(n-1)\bar K^2               }\\
     < &\frac{2( n-2 ) }{(n-1) (n+2\sqrt{6n}+4)             }<\frac{2n-5}{(n-1)(n+2)}.
 \end{align*}

\begin{equation*}
  2x  \mathring{a}''  + \mathring{a}'  <\frac{2n-5}{(n-1)(n+2)}+\frac{1}{n(n-1)}= \frac{2
    ( n-1 )}{n ( n+2 )}.
\end{equation*}

(iii)
\begin{align*}
   & \frac{n-2}{\sqrt{n(n-1)}}\sqrt{x\mathring{a}}+\mathring{a}  < \frac{x}{n}+n\bar K \\
  \Leftrightarrow & \frac{(n-2)^2}{n(n-1)}\Big(x(a-\frac{x}{n})\Big)  < (\frac{2x}{n}-a+n\bar K)^2\\
  \Leftrightarrow & a\Big(\frac{x}{n-1}+2\bar K\Big)  < \Big(\frac{x}{n-1}+2\bar K\Big)^2+(n-2)\bar K^2\\
  \Leftrightarrow & a  < \Big(\frac{x}{n-1}+2\bar K\Big)+\frac{(n-2)\bar K^2}{\frac{x}{n-1}+2\bar K}.
\end{align*}

 (iv)
 \begin{align*}
   &n\bar K(\mathring{a}+x\mathring{a}')-a( \mathring{a}-x\mathring{a}') \\
   =& \frac{2n\bar K}{n-1}\bigg( \frac{x^2+3(n-1)\bar K x +n(n-1)^2\bar K^2}{\sqrt{x^2+4(n-1)\bar K x+2n(n-1)^2\bar K^2}}-x-(n-1)\bar K\bigg)    \\
   =&  \frac{2n(n-2)(n-1)\big(2x+n(n-1)\bar K\big)\bar K^4}{\big(x^2+3(n-1)\bar K x+n(n-1)^2\bar K^2\big)a+\big(x+(n-1)\bar K\big)(n-1)a^2}  \\
   \geq&\frac{2n(n-2)(n-1)^2\bar K^4}{\big(x+\sqrt{2+\sqrt{2n}}(n-1)\bar K\big)^2}.
 \end{align*}
The last inequality above is equivalent to
  \begin{align*}
 \Leftrightarrow & x^3+\Big(4\sqrt{2+\sqrt{2n}}+n-5\Big)(n-1)\bar K x^2\\
 &+2\Big(n\sqrt{2+\sqrt{2n}}+\sqrt{2n}-n\Big)(n-1)^2\bar K^2 x+\sqrt{2n}n(n-1)^3\bar K^3\\
  &\geq\big(x+\sqrt{2n}(n-1)\bar K\big)\big(x^2+3(n-1)\bar K x+n(n-1)^2\bar K^2\big)\\
  &\geq(n-1)a\cdot\big(x^2+3(n-1)\bar K x+n(n-1)^2\bar K^2\big).
 \end{align*}

(v)
set $A=(n-1)a=\sqrt{x^2+4(n-1)\bar K x+2n(n-1)^2\bar K^2}.$
\begin{align*}
   &  2\mathring{a}-\frac{x}{n}+x\mathring{a}'\\
    =&\frac{3x^2+10(n-1)\bar K x +4n(n-1)^2\bar K^2}{A}-\frac{4x}{n}\\
     \leq& 2\sqrt{2n}\bar K-\frac{(n-4)x}{n(n-1)}-\frac{6(\sqrt{2n}-2)\bar K x}{3x+2\sqrt{2n}(n-1)\bar K}.
\end{align*}
The inequality above is equivalent to
\begin{align*}
\Leftrightarrow & 3x+2\sqrt{2n}(n-1)\bar K-\frac{3x^2+10(n-1)\bar K x +4n(n-1)^2\bar K^2}{A}\\
&=\frac{\big(3x+2\sqrt{2n}(n-1)\bar K\big)^2A^2-\big(3x^2+10(n-1)\bar K x +4n(n-1)^2\bar K^2\big)^2}{A^2\big(3x+2\sqrt{2n}(n-1)\bar K\big)+A\big(3x^2+10(n-1)\bar K x +4n(n-1)^2\bar K^2\big)}\\
&\geq\frac{12(\sqrt{2n}-2)(n-1)\bar K x\cdot A}{A\big(3x+2\sqrt{2n}(n-1)\bar K\big)+\big(3x^2+10(n-1)\bar K x +4n(n-1)^2\bar K^2\big)}\\
&\geq\frac{6(\sqrt{2n}-2)(n-1)\bar K x}{3x+2\sqrt{2n}(n-1)\bar K}.
 \end{align*}
The last inequality above is equivalent to
\begin{align*}
\Leftrightarrow & A\big(3x+2\sqrt{2n}(n-1)\bar K\big)\\
&\geq3x^2+10(n-1)\bar K x +4n(n-1)^2\bar K^2.
 \end{align*}

 (vi)
 \begin{align*}
   &  \frac{x}{n-1}\left( a   +n\bar K
    \right) - (\frac{x}{n-1}+2n\bar K) \left(  a+\mathring{a}
      -n\bar K - x\mathring{a}' \right) \\
    =& \frac{x}{n-1}\left( 2n\bar K + x\mathring{a}' -\mathring{a}     \right)
    +2n\bar K\left( n\bar K + x\mathring{a}'-a-\mathring{a}     \right)\\
    <&\frac{x}{n-1}\left(   2n\bar K + \frac{x}{n-1}  -\Big(\frac{x}{n-1}+2\bar K\Big)\right)\\
    &+2n\bar K\left( n\bar K + \frac{x}{n-1} -2\Big(\frac{x}{n-1}+2\bar K\Big)    \right)\\
    <& -\frac{2x\bar K}{n-1}+2n(n-4)\bar K^2.
 \end{align*}

\end{proof}

$ \ $
\section{Preservation of curvature pinching}

In this section, we prove the curvature pinching condition preserves along the mean curvature flow in spheres. First we consider the submanifold $M^n (n\geq7)$ of arbitrary codimension in the sphere ${S}^{n+p}(\frac{1}{\sqrt{\bar K}})$. For the mean curvature flow of hypersurfaces $M^n (n\geq3)$ in spheres, the proof is similar.  First, we prove the following inequality.

\begin{lemma}\label{negative}For $x\geq0$, the following inequality holds.
\begin{align*}
a( \mathring{a}-x\mathring{a}')-n\bar K(\mathring{a}+x\mathring{a}')
       +P_{2} \left( 2 \mathring{a} - \frac{x}{n}  + x \mathring{a}' \right)- \frac{3}{2} P^2_{2}<0. \\
\end{align*}
\end{lemma}
\begin{proof}
Without loss of generality, we assume that $\bar K=1$. Its discriminant satisfies
\begin{align}\label{n7}
  \Delta=&\left( 2 \mathring{a} - \frac{x}{n}  + x \mathring{a}' \right)-\sqrt{6\big(n(\mathring{a}+x\mathring{a}')-a( \mathring{a}-x\mathring{a}')\big)}. \\
   =&\frac{3x^2+10(n-1) x +4n(n-1)^2}{(n-1)\sqrt{x^2+4(n-1) x+2n(n-1)^2}}-\frac{4x}{n}\notag\\
   &-\sqrt{\frac{12n}{n-1}\bigg( \frac{x^2+3(n-1) x +n(n-1)^2}{\sqrt{x^2+4(n-1) x+2n(n-1)^2}}-x-(n-1)\bigg) }.\notag
\end{align}
 From Lemma \ref{app} (iv),(v), we need to prove the following inequality holds.
 \begin{align}\label{n8}
  &\Big(x+\sqrt{2+\sqrt{2n}}(n-1)\Big)\Big(2\sqrt{2n}-\frac{(n-4)x}{n(n-1)}-\frac{6(\sqrt{2n}-2) x}{3x+2\sqrt{2n}(n-1)}\Big)\\
  <& 2\sqrt{3n(n-2)}(n-1).\notag
\end{align}

\begin{align}
  &\Big(x+\sqrt{2+\sqrt{2n}}(n-1)\Big)\Big(2\sqrt{2n}-\frac{(n-4)x}{n(n-1)}-\frac{6(\sqrt{2n}-2) x}{3x+2\sqrt{2n}(n-1)}\Big)\\
  <& \Big(x+\sqrt{2+\sqrt{2n}}(n-1)\Big)\Big(2\sqrt{2n}-\frac{(n-4)x}{n(n-1)}\Big)\notag\\
  & -2(\sqrt{2n}-2)x\cdot\frac{3\sqrt{2+\sqrt{2n}}}{2\sqrt{2n}}\notag\\
  <& 2\sqrt{3n(n-2)}(n-1).\notag
  \end{align}
The last inequality above is equivalent to
\begin{align}
 \Leftrightarrow &-\frac{n-4}{n(n-1)}x^2+\left(2\sqrt{2n}-\Big(4-\frac{4}{n}-3\sqrt{\frac2n}\Big)\sqrt{2+\sqrt{2n}}\right)x  \\
  &-2\sqrt{n}(n-1)\left(\sqrt{3(n-2)}- \sqrt{4+2\sqrt{2n}} \right)<0.\notag
\end{align}

Now we prove its discriminant is negative, i.e.,
\begin{align}
  &1-\Big(2-\frac{2}{n}-\frac{3}{\sqrt{2n}}\Big)\sqrt{\frac{2+\sqrt{2n}}{2n}} <\sqrt{\frac{n-4}{n\sqrt{n}}\Big(\sqrt{3(n-2)}- \sqrt{4+2\sqrt{2n}} \Big)}.
\end{align}

(i)$n \geq 66$
\begin{align*}
  & 1+\frac{4}{n-4}+\sqrt{\frac{4+2\sqrt{2n}}{n}}  <\sqrt{3\Big(1-\frac2n\Big)} \\
  \Leftrightarrow 1 & < \frac{n-4}{n\sqrt{n}}\Big(\sqrt{3(n-2)}- \sqrt{4+2\sqrt{2n}} \Big).
\end{align*}

(ii) $13 \leq n \leq 65$

Since
\begin{align*}
  &1-\Big(2-\frac{2}{n}-\frac{3}{\sqrt{2n}}\Big)\sqrt{\frac{2+\sqrt{2n}}{2n}} \\  &<1-\Big(2-\frac{2}{13}-\frac{3}{\sqrt{26}}\Big)\sqrt{\frac{2+\sqrt{130}}{130}}\\
   &< \sqrt{0.3555},
\end{align*}
then we have
\begin{align*}
  & 0.3555\Big(1+\frac{4}{n-4}\Big)+\sqrt{\frac{4+2\sqrt{2n}}{n}}  <\sqrt{3\Big(1-\frac2n\Big)} \\
  &\Leftrightarrow 0.3555  < \frac{n-4}{n\sqrt{n}}\Big(\sqrt{3(n-2)}- \sqrt{4+2\sqrt{2n}} \Big). \end{align*}

(iii) $9 \leq n \leq 12$

We can check that
\begin{equation*}
  1-\Big(2-\frac{2}{n}-\frac{3}{\sqrt{2n}}\Big)\sqrt{\frac{2+\sqrt{2n}}{2n}} < \sqrt{0.1814},
\end{equation*}
then we have
\begin{align*}
    0.1814  < \frac{n-4}{n\sqrt{n}}\Big(\sqrt{3(n-2)}- \sqrt{4+2\sqrt{2n}} \Big).  \end{align*}

When $n=8$, (\ref{n8}) is
\begin{align*}
   & 8-\frac{x}{14}-\frac{12x}{3x+56}-\frac{168}{x+7\sqrt{6}}<0 \\
   & \Leftrightarrow \frac{3}{14}x^3-\Big(8-\frac{3\sqrt{6}}{2}\Big)x^2-56(\sqrt{6}-1)x+56^2(3-\sqrt{6})>0.
\end{align*}
We calculate the minimum is
\begin{align*}
    &\min\{\frac{3}{14}x^3-\Big(8-\frac{3\sqrt{6}}{2}\Big)x^2-56(\sqrt{6}-1)x+56^2(3-\sqrt{6})\}\\
   &=86.697>0,\ x=19.827.
\end{align*}

When $n=7$, (\ref{n7}) is
\begin{equation*}
  \frac{x^2+20 x +336}{2\sqrt{x^2+24 x+504}}-\frac{4x}{7}-\sqrt{14\Big( \frac{x^2+18 x +252}{\sqrt{x^2+24 x+504}}-x-6\Big)}<0.
\end{equation*}
We calculate the maximum by numerical evaluation.
\begin{align*}
   &\max\Big\{\frac{x^2+20 x +336}{2\sqrt{x^2+24 x+504}}-\frac{4x}{7}-\sqrt{14\Big( \frac{x^2+18 x +252}{\sqrt{x^2+24 x+504}}-x-6\Big)}\Big\}\notag\\
   &=-0.264<0, \  x = 20.399.
\end{align*}

\end{proof}

We denote $\mathring{a}' ( | H |^{2})$ and $\mathring{a}'' ( | H
|^{2} )$ by $\mathring{a}'$ and
$\mathring{a}''$, respectively. Then the evolution
equation of $\mathring{a}$ satisfies
\begin{equation}\label{a}
  \frac{\partial}{\partial t} \mathring{a} = \Delta \mathring{a} +2 \mathring{a}' \cdot
  ( - | \nabla H |^{2} +R_{2} +n \bar K | H |^{2} ) - \mathring{a}'' \cdot |
  \nabla | H |^{2} |^{2} .
\end{equation}

Since $M$ is
compact, there exists a small positive number $0<\epsilon\ll\frac{1}{n^2},$
such that $M$ satisfies
\begin{equation}\label{con-w}
  |\mathring{A}|^{2} < \mathring{a} - \epsilon\omega, \ \    \ \
   \omega = \frac{| H |^{2}}{n-1} +2n\bar K>a.
\end{equation}
Then we prove that the pinching condition above is
preserved along the flow.

\begin{proposition}\label{baochi}
Let $F_0 :M \rightarrow \mathbb{S}^{n+p}(\frac{1}{\sqrt{\bar K}})$ be a compact submanifold immersed in the sphere. Suppose there exists a small positive number $\epsilon(\ll\frac{1}{n^2})$ such that $ |A| ^ 2 \leq a(|H|^2)-\epsilon\omega, $
 then this condition holds
  along the mean curvature flow for all time $t \in [ 0,T )$ where $T\leq\infty$.
\end{proposition}

\begin{proof}
  Let $U= |\mathring{A}|^{2} - \mathring{a}+\epsilon\omega$.
    By Lemmas \ref{ee},\ref{eq} and (\ref{a}), we have
  \begin{eqnarray*}
    \Big( \frac{\partial}{\partial t} - \Delta \Big) U & = & -2 | \nabla \mathring A |^{2} +2 (\mathring{a}'-\frac{\epsilon}{n-1})   | \nabla H |^{2} +
    \mathring{a}''   | \nabla | H |^{2} |^{2}\\
    &  & +2R_{1} - \frac{2}{n} R_{2} -2 n \bar K|\mathring{A}|^{2} -2
    (\mathring{a}'-\frac{\epsilon}{n-1}) \cdot ( R_{2} +n \bar K| H |^{2} ) \\
    & \leq &2 \Big( - \frac{2 ( n-1 )}{n (
    n+2 )} + \mathring{a}'-\frac{\epsilon}{n-1} +2 | H |^{2}
    \mathring{a}'' \Big) | \nabla H |^{2}\\
    &  & +2 |\mathring{A}|^{2} \Big( |\mathring{A}|^{2} + \frac{1}{n} | H |^{2} - n \bar K
    \Big) +2P_{2} \Big( 2 |\mathring{A}|^{2} - \frac{1}{n} | H |^{2} - \frac{3}{2}
    P_{2} \Big)\\
    &  & -2 (\mathring{a}'-\frac{\epsilon}{n-1}) \cdot | H |^{2} \Big(
    |\mathring{A}|^{2} + \frac{1}{n} | H |^{2} +n\bar K-P_{2} \Big)  .
  \end{eqnarray*}
  From Lemma \ref{app} (ii), the coefficient of $| \nabla H |^{2}$ is
  negative.

  Replacing $|\mathring{A}|^{2}$ by $U+ \mathring{a}-\epsilon\omega$,
  the above formula becomes
  \begin{eqnarray*}
    \left( \frac{\partial}{\partial t} - \Delta \right) U & \leq &  2U \left( 2 \mathring{a} + \frac{1}{n} | H |^{2} -n \bar K  -
    \mathring{a}' | H |^{2}  +2P_{2} + \epsilon \big( \frac{| H |^{2}}{n-1} -2 \omega
    \big) \right)  \nonumber\\
    &  & +2U^{2}+2 \bigg( a( \mathring{a}-| H |^{2}\mathring{a}')-n\bar K(\mathring{a}+| H |^{2}\mathring{a}') \bigg)
    \nonumber\\
    &  & +2P_{2} \left( 2 \mathring{a} - \frac{1}{n} | H |^{2} + | H
    |^{2} \mathring{a}' - \epsilon \big( \frac{| H |^{2}}{n-1} +2 \omega \big)- \frac{3}{2} P_{2} \right)
    \label{Uandneg}\\
    &  & +2 \epsilon    \left( \frac{ |
    H |^{2}}{n-1} \cdot \left( a +n \bar K
    \right) - \omega\left( a+ \mathring{a} -n \bar K - \mathring{a}' | H |^{2} \right) \right) \nonumber\\
    &  & +2 \epsilon^{2}   \omega \big(\omega- \frac{| H |^{2}}{n-1} \big) .
  \end{eqnarray*}
From Lemma \ref{app} (vi) and Lemma \ref{negative}, we have
\begin{eqnarray*}
 &&  a( \mathring{a}-| H |^{2}\mathring{a}')-n\bar K(\mathring{a}+| H |^{2}\mathring{a}')
    \nonumber\\
    &  &+ P_{2} \left( 2 \mathring{a} - \frac{1}{n} | H |^{2} + | H
    |^{2} \mathring{a}' - \epsilon \big( \frac{| H |^{2}}{n-1} +2 \omega \big)- \frac{3}{2} P_{2} \right)
    \label{Uandneg}\\
    &  &+  \epsilon    \left( \frac{ |
    H |^{2}}{n-1} \cdot \left( a +n \bar K
    \right) - \omega\left( a+ \mathring{a} -n \bar K - \mathring{a}' | H |^{2} \right) + 2n\bar K\epsilon  \omega \right)\\
    &<& a( \mathring{a}-| H |^{2}\mathring{a}')-n\bar K(\mathring{a}+| H |^{2}\mathring{a}')
    \nonumber\\
    &  &+ P_{2} \left( 2 \mathring{a} - \frac{1}{n} | H |^{2} + | H
    |^{2} \mathring{a}' - \frac{3}{2} P_{2} \right)
    \label{Uandneg}\\
    &  &+  2\epsilon  \bar K  \left(n\bar K(n-4+2n\epsilon) -\frac{| H |^{2}}{n-1}(1-n\epsilon) \right)\\
    &<&0.
\end{eqnarray*}
Then the assertion follows from the maximum
  principle.
\end{proof}

$ \ $
\section{Convergence theorem in the sphere}
In this section, we prove the convergence theorem for the mean curvature flow of submanifolds $M^n(n\geq7)$ in $\mathbb{S}^{n+p}(\frac{1}{\sqrt{\bar K}})$. First, we derive an estimate for the traceless second fundamental
form, which guarantees that $M$ becomes
spherical along the mean curvature flow.

\begin{proposition}\label{san1}
There exist constants $C < \infty$ and $\sigma>0$ both depending only on the initial surface such that for all time $t \in [0,T)$ where $T\leq\infty$, we have the estimate

\begin{align}\label{pinch}
|\mathring A|^ 2  \leq  C (|H| ^2+\bar{K})^{1 -\sigma}e^{-2\sigma t}.
\end{align}
\end{proposition}
To obtain Theorem \ref{pinch}, we need to find the upper bound of $$f_{\sigma} := \frac{|\mathring A|^2}{\mathring{a}^{1-\sigma}}$$  with the help of  the Stampacchia iteration as in \cite{B}. Here, we get the evolution equation of $f_{\sigma}$ with two parameters.
\begin{lemma}\label{yanhua1}
For every $\sigma\in (0,1)$, we have the evolution equation
\begin{eqnarray}\label{inequ}
\frac \partial { \partial t } f _{ \sigma} &\leq& \Delta f _\sigma + \frac{ 2| \nabla  f _ \sigma | |\nabla  H| }{|\mathring A| }- \frac { 2 \epsilon f_{\sigma}| \nabla H | ^ 2 } { n|\mathring A|^2 }  \\
&&+ 6 \sigma | A| ^ 2 f _ \sigma-2(n\sigma+\epsilon)\bar K f_\sigma.\nonumber
\end{eqnarray}
\end{lemma}
\begin{proof}
From Lemma \ref{ee}, we have
\begin{align}\label{shijian}
\frac{\partial }{\partial t}f_{\sigma} &= f_{\sigma}\Big(\frac{\frac{\partial }{\partial t}|\mathring A|^2 }{ |\mathring A|^2 }-(1-\sigma)\frac{\frac{\partial }{\partial t}\mathring a }{ \mathring a }\Big).    \\
       &\notag
\end{align}
By direct computation, we have
\begin{align}\label{daoshu}
\Delta f_{\sigma} &=f_{\sigma}\Big( \frac{ \Delta  |\mathring A|^2  }{ |\mathring A|^2 }-(1-\sigma)\frac{\Delta\mathring a }{ \mathring a }\Big) \\
        &\quad -\frac{ 2(1-\sigma) }{\mathring a } \big\langle \nabla\mathring a, \nabla f_{\sigma} \big\rangle  +\frac{ \sigma(1-\sigma) }{ \mathring a^2 } f_{\sigma}|{ \nabla\mathring a }|^2. \notag
\end{align}
Combing (\ref{shijian}) and (\ref{daoshu}) and from Lemma \ref{ee}-\ref{app}, we have
\begin{align*}\label{fangcheng}
\Big(\frac{\partial}{\partial t}-\Delta\Big) f_{\sigma} &= \frac{ 2(1-\sigma) }{\mathring a } \big\langle \nabla\mathring a, \nabla f_{\sigma} \big\rangle  -\frac{ \sigma(1-\sigma) }{ \mathring a^2 } f_{\sigma}|{ \nabla\mathring a }|^2 \\
        &\quad + f_{\sigma}\Big(\frac{1-\sigma}{\mathring a}( 2\mathring{a}'|\nabla H|^2+\mathring{a}''|\nabla|H|^2|^2)-\frac{2|\nabla\mathring A|^2}{|\mathring A|^2}\Big)\\
        &\quad+ 2 f_{\sigma} \Bigg( \frac{R_1-\frac1n R_2-n\bar K |\mathring A|^2  }{|\mathring A|^2} - \frac{(1-\sigma)\mathring{a}'}{\mathring a}  (R_2+n \bar K |H|^2)\Bigg)\notag\\
        &\leq \frac{ 4\mathring{a}'|H| }{\mathring a } |\nabla H|| \nabla f_{\sigma}|+ 2f_{\sigma}\frac{2(n-1)}{n(n+2)}\cdot\frac{|\mathring A|^2-\mathring a}{|\mathring A|^2\mathring a}|\nabla H|^2 \\
        &\quad+ 2\sigma f_{\sigma} \Bigg( \frac{R_1-\frac1n R_2  }{|\mathring A|^2} -n\bar K\Bigg) \\
        &\quad+ 2 (1-\sigma)f_{\sigma} \Bigg( \frac{R_1-\frac1n R_2  }{|\mathring A|^2} -n\bar K- \frac{\mathring{a}'}{\mathring a}  (R_2+n \bar K |H|^2)\Bigg)\\
        &\leq \frac{ 2| \nabla f_{\sigma}||\nabla H| }{|\mathring A| } - \frac{2\epsilon f_{\sigma}|\nabla H|^2}{n|\mathring A|^2} + 2\sigma f_{\sigma} \Big( 3|A|^2 -n\bar K\Big) \\
        &\quad+ 2 (1-\sigma)f_{\sigma} \Bigg( \frac{R_1-\frac1n R_2  }{|\mathring A|^2} -n\bar K- \frac{\mathring{a}'}{\mathring a}  (R_2+n \bar K |H|^2)\Bigg).\notag
\end{align*}
The last term in the bracket satisfies
\begin{align}
&\frac{R_1-\frac1n R_2  }{|\mathring A|^2} -n\bar K- \frac{\mathring{a}'}{\mathring a}  (R_2+n \bar K |H|^2)\\ \notag
   \leq &\quad|A|^2+2P_2-n\bar K-\frac{1}{|\mathring A|^2}\Big(\frac32 P^2_2+\frac{|H|^2}{n}P_2\Big)- \frac{\mathring{a}'|H|^2}{\mathring a}  (|A|^2-P_2+n \bar K )\\ \notag
   \leq&\quad \frac{|A|^2}{\mathring a}(\mathring a-\mathring{a}'|H|^2)-\frac{n\bar K}{\mathring a}(\mathring a+\mathring{a}'|H|^2)+\frac{P_2}{\mathring a}\Big(2\mathring a+\mathring{a}'|H|^2-\frac{|H|^2}{n}-\frac32 P_2\Big)\\ \notag
   \leq&\quad \frac{1}{\mathring a}\Bigg(a(\mathring a-\mathring{a}'|H|^2)-n\bar K(\mathring a+\mathring{a}'|H|^2)+P_2\Big(2\mathring a+\mathring{a}'|H|^2-\frac{|H|^2}{n}-\frac32 P_2\Big)\Bigg)\\ \notag
   &-\frac{\epsilon \omega}{\mathring a}(\mathring a-\mathring{a}'|H|^2).\notag
\end{align}
The term in the big bracket of the last inequality is non-positive under our pinching assumption.
We complete the lemma with the following inequality.
\begin{equation*}
  -\frac{\epsilon \omega}{\mathring a}(\mathring a-\mathring{a}'|H|^2)\leq-\frac{2\epsilon \omega \bar K}{\mathring a}\leq-2\epsilon \bar K.
\end{equation*}
\end{proof}

Since the term $\sigma |A|^2 f_\sigma$ in the evolution equation is positive, we cannot use the ordinary maximum principle. As in \cite{B,H3}, we need the negative gradient terms to proceed the iteration. Applying the Simons identity \cite{Si}, we get
\begin{align}
\frac 12 \Delta | \mathring A | ^ 2 & =  \langle \mathring A , \nabla^2 H\rangle + | \nabla \mathring A |^2+ n\bar K|\mathring A|^2-R_1+R_3. \label{Simons}
\end{align}

Now we have the following estimate.
\begin{lemma}\label{Z2}Let $F :M \rightarrow \mathbb{S}^{n+p}(\frac{1}{\sqrt{\bar K}})$ be a compact submanifold immersed in the sphere with
constant curvature $\bar K$. If $F$ satisfies pinching condition (\ref{con-w}), then
\begin{equation}
  n\bar K|\mathring A|^2-R_1+R_3\geq
           \frac{n}{2}|\mathring A|^2(\epsilon|A|^2-\sqrt{2n}\bar{K}).
\end{equation}
\end{lemma}
\begin{proof}
From Lemma \ref{eq}, we have
\begin{align*}
   n\bar K|\mathring A|^2-R_1+R_3&\geq\frac{n}{2}|\mathring A|^2\Big(\frac{|H|^2}{n(n-1)}+2\bar K\Big)-\frac n2|\mathring A|^4\\
   &=\frac{n}{2}|\mathring A|^2\Big(\frac{|H|^2}{n-1}+2\bar K-|A|^2\Big)\\
   &\geq\frac{n}{2}|\mathring A|^2\Big(\frac{|H|^2}{n-1}+2\bar K-a+\epsilon\omega\Big)\\
   &\geq\frac{n}{2}|\mathring A|^2\Big(\epsilon\omega-(\sqrt{2n}-2)\bar K\Big).
\end{align*}

\end{proof}

Then we can get the required Poincar\'{e} inequality.
\begin{lemma}\label{Poincare1}
For every $p\geq 4$ and $\eta>0$ we have the estimate
\begin{eqnarray}
&&\int_{M_t}  |A|^2f_{\sigma} d \mu _t  \\
&\leq &\int_{M_{t}} \left( \frac{2(p-1)}{n\epsilon\eta}f_{\sigma}^{p-2}|\nabla f_{\sigma} |^2
  + \frac{2(p-1)\eta+5}{n\epsilon}\cdot\frac{f_{\sigma}^{p}|
  \nabla H |^{2}}{|\mathring{A}|^{2}}  \right) d \mu_{t}\nonumber\\
  &&+ \frac{\sqrt{2n}}{\epsilon}\int_{M_t}  \bar Kf ^ p _ \sigma d \mu _t .\nonumber
\end{eqnarray}
\end{lemma}
\begin{proof}
From Lemma \ref{Z2}, we have
\begin{align}\label{f1}
        \Delta f_{\sigma} \geq & \frac{ \Delta  |\mathring A|^2  }{ \mathring a^{1-\sigma} }-(1-\sigma)\frac{f_\sigma\Delta\mathring a }{ \mathring a }-\frac{ 2(1-\sigma)\mathring a' }{\mathring a } \big\langle \nabla |H|^2, \nabla f_{\sigma} \big\rangle  \\ \notag
        \geq& \frac{1}{\mathring a^{1-\sigma}}\Big(2\langle \mathring A,  \nabla^2 H \rangle +n|\mathring A|^2(\epsilon|A|^2-\sqrt{2n}\bar{K})  \Big)\\  \notag
        &-\frac{ 2| \nabla  f _ \sigma | |\nabla  H| }{|\mathring A| }-(1-\sigma)\frac{f_\sigma\Delta\mathring a }{ \mathring a }.\notag
    \end{align}

\begin{align}\label{f11}
        nf_\sigma(\epsilon|A|^2-\sqrt{2n}\bar{K})\leq& \Delta f_{\sigma}-\frac{2\langle \mathring A,  \nabla^2 H \rangle}{\mathring a^{1-\sigma}}
       +(1-\sigma)\frac{f_\sigma\Delta\mathring a }{ \mathring a }\\ \notag
       & + \frac{ 2| \nabla  f _ \sigma | |\nabla  H| }{|\mathring A| }.\notag
    \end{align}
We multiply the equation (\ref{f11}) by $f_\sigma^{p-1}$ and integrate on both sides, then
\begin{eqnarray}\label{p1}
  &&\int_{M_t} nf^p_\sigma(\epsilon|A|^2-\sqrt{2n}\bar{K}) d\mu_t  \\
   &\leq& -\int_{M_t}\frac{2f_\sigma^{p-1}\langle \mathring A,  \nabla^2 H \rangle}{\mathring a^{1-\sigma}}d\mu _ t  + \int_{M_t} (1-\sigma)\frac{f^p_\sigma\Delta\mathring a }{ \mathring a }d\mu_t \nonumber \\
   &&  +\int_{M_t}\frac{ 2f_\sigma^{p-1}| \nabla  f _ \sigma | |\nabla  H| }{|\mathring A| }d \mu_ t. \nonumber
\end{eqnarray}
The Codazzi equation implies $\nabla_{i} \mathring{A}^{\alpha}_{i j} =
\frac{n-1}{n}   \nabla_{j} H^{\alpha}$. From this formula and Lemma \ref{app} (i), we get
\begin{eqnarray}
  & &\ \ \ - \int_{M_{t}} \frac{f_{\sigma}^{p-1} \langle \mathring{A} , \nabla^{2}
  H \rangle}{\mathring{a}^{1- \sigma}}   d\mu_{t} \\
    &  &=  \int_{M_{t}} \nabla_{i} \left( \frac{f_{\sigma}^{p-1} \mathring{A}_{i j}^{\alpha}}{\mathring{a}^{1-
  \sigma}} \right) \nabla_{j} H^{\alpha} d\mu_{t} \label{eC2f2}\nonumber\\
  &  &= \int_{M_{t}} \left( \frac{( p-1 )
  f_{\sigma}^{p-2}\mathring{A}^{\alpha}_{i j}}{\mathring{a}^{1- \sigma}}
  \nabla_{i} f_{\sigma}
   - \frac{( 1- \sigma ) f_{\sigma}^{p-1}\mathring{A}^{\alpha}_{i j} }{\mathring{a}^{2- \sigma}}
    \nabla_{i} \mathring{a}  +
  \frac{f_{\sigma}^{p-1}}{\mathring{a}^{1- \sigma}} \nabla_{i}
  \mathring{A}^{\alpha}_{i j} \right) \nabla_{j} H^{\alpha} d \mu_{t} \nonumber\\
    &&\leq  \int_{M_{t}} \left( ( p-1 )\frac{ f_{\sigma}^{p-1}| \nabla
  f_{\sigma} || \nabla H |}{| \mathring{A} |}
   +\frac{2f_{\sigma}^{p}| \nabla H |^{2}}{| \mathring{A} |^{2}}  \right) d \mu_{t}. \nonumber
\end{eqnarray}
and
\begin{eqnarray}\label{p2}
  \int_{M_{t}} \frac{f_{\sigma}^{p}}{\mathring{a}} \Delta
  \mathring{a}   d \mu_{t} & = & - \int_{M_{t}} \left\langle \nabla
  \left( \frac{f_{\sigma}^{p}}{\mathring{a}} \right) , \nabla
  \mathring{a} \right\rangle d \mu_{t} \\ \nonumber
  & = & \int_{M_{t}} \left( - \frac{p f_{\sigma}^{p-1}}{\mathring{a}}
  \left\langle \nabla f_{\sigma} , \nabla \mathring{a} \right\rangle +
  \frac{f_{\sigma}^{p}}{\mathring{a}^{2}} \left| \nabla \mathring{a}
  \right|^{2} \right) d \mu_{t}\label{inte2}\\ \nonumber
  & \leq & \int_{M_{t}} \left( \frac{p f_{\sigma}^{p-1}|
  \nabla f_{\sigma} | |   \nabla H |}{|\mathring{A}|}  + \frac{f_{\sigma}^{p}|
  \nabla H |^{2}}{|\mathring{A}|^{2}}  \right) d \mu_{t} . \nonumber
\end{eqnarray}

Putting (\ref{p1})-(\ref{p2}) together, we obtain
\begin{eqnarray}
  &&\int_{M_t} nf^p_\sigma(\epsilon|A|^2-\sqrt{2n}\bar{K}) d\mu_t  \\
   &\leq& \int_{M_{t}} \left( \frac{3p f_{\sigma}^{p-1}|
  \nabla f_{\sigma} | |   \nabla H |}{|\mathring{A}|}  + \frac{5f_{\sigma}^{p}|
  \nabla H |^{2}}{|\mathring{A}|^{2}}  \right) d \mu_{t} \nonumber\\
  &\leq& \int_{M_{t}} \left( \frac{2(p-1)}{\eta}f_{\sigma}^{p-2}|\nabla f_{\sigma} |^2
  + \Big(2(p-1)\eta+5\Big)\frac{f_{\sigma}^{p}|
  \nabla H |^{2}}{|\mathring{A}|^{2}}  \right) d \mu_{t}. \nonumber
\end{eqnarray}

\end{proof}

Now we show that the $L^p$-norm of $f_\sigma$ is bounded for sufficiently large $p$.
\begin{lemma}\label{lem1}For any $p\geq\frac{n^3}{32\epsilon}+1$ and
$\sigma\leq \frac{n\epsilon\sqrt{\epsilon}}{24\sqrt{p-1}}$,
there exist a constant $C$ depending only on the initial surface such that for all
$t\in [0, T)$ where $T\leq\infty$, we have
\begin{eqnarray}\label{8-ineq}
\bigg(\int_{M_t}f_\sigma^pd\mu_t\bigg)^{\frac{1}{p}}\leq Ce^{-2\sigma\bar{K}t}.
\end{eqnarray}
\end{lemma}

\begin{proof}For $t\geq t_0$, form Lemma \ref{yanhua1}, we have
\begin{eqnarray}\label{9-ineq1}
  \frac{\partial}{\partial t}\int_{M_t}f_\sigma^pd\mu_t &\leq& \int_{M_t}pf_\sigma^{p-1}\frac{\partial}{\partial t}f_\sigma
d\mu_t \\
   &\leq& -p(p-1)\int_{M_t}f_\sigma^{p-2}|\nabla
f_\sigma|^2d\mu_t \nonumber \\
   && 2p\int_{M_t}\frac{ f^{p-1} _ \sigma| \nabla  f_ \sigma | |\nabla  H| }{|\mathring A| } d\mu_t -2p\epsilon
\int_{M_t}\frac{f_{\sigma}^{p}|\nabla H|^2}{|\mathring A|^2}d\mu_t \nonumber\\
   && +6\sigma p\int_{M_t}|A|^2f_\sigma^pd\mu_t-2p(n\sigma+\epsilon)\bar{K}\int_{M_t}f_\sigma^{p}d\mu_t. \nonumber
\end{eqnarray}
In view of the pinching condition we can estimate
\begin{eqnarray}\label{10-ineq1}
   && 2p\int_{M_t}\frac{ f^{p-1} _ \sigma| \nabla  f_ \sigma | |\nabla  H| }{|\mathring A| } d\mu_t \\
   &\leq& \frac{p}{\mu}\int_{M_t}f_\sigma^{p-2}|\nabla
f_\sigma|^2d\mu_t+p\mu\int_{M_t}\frac{f_{\sigma}^{p}|\nabla H|^2}{|\mathring A|^2}d\mu_t. \nonumber
\end{eqnarray}
Substituting (\ref{10-ineq1}) to (\ref{9-ineq1}),  letting
$\mu=\frac{2}{p-1}$ and
$p\geq\frac{2}{\epsilon}+1$  we obtain
\begin{eqnarray}
  \frac{\partial}{\partial t}\int_{M_t}f_\sigma^pd\mu_t  &\leq&
-\frac{p(p-1)}{2}\int_{M_t}f_\sigma^{p-2}|\nabla
f_\sigma|^2d\mu_t \\
   &&  -p\epsilon
\int_{M_t}\frac{f_{\sigma}^{p-1}}{|H|^{2(1-\sigma)}}|\nabla
A|^2d\mu_t\nonumber \\
   &&  +6\sigma p\int_{M_t}|A|^2f_\sigma^pd\mu_t-2p(n\sigma+\epsilon)\bar{K}\int_{M_t}f_\sigma^{p}d\mu_t. \nonumber
\end{eqnarray}
This together with Lemma \ref{Poincare1} implies
\begin{eqnarray}\label{11-ineq1}
  &&\frac{\partial}{\partial t}\int_{M_t}f_\sigma^pd\mu_t \\
  &\leq& -p(p-1)\bigg(\frac{1}{2}-\frac{12\sigma
}{n\eta \epsilon
}\bigg)\int_{M_t}f_\sigma^{p-2}|\nabla
f_\sigma|^2d\mu_t \nonumber\\
   && -p\bigg(\epsilon-\frac{6\sigma \big(2(p-1)\eta+5\big)}{n\epsilon}\bigg)
\int_{M_t}\frac{f_{\sigma}^{p-1}}{(\alpha|H|+\beta \bar{K})^{1-\sigma}}|\nabla
A|^2d\mu_t \nonumber\\
   && -2\Big(\epsilon-\frac{3\sigma\sqrt{2n}}{\epsilon}\Big) p\bar{K}\int_{M_t}f_\sigma^{p}d\mu_t-2n\sigma p\bar{K}\int_{M_t}f_\sigma^{p}d\mu_t. \nonumber
\end{eqnarray}
Now we pick
$\eta=\frac{24\sigma}{n\epsilon}$, $p\geq\frac{n^3}{32\epsilon}+1$ and let
$$\sigma\leq
\min\big\{\frac{\epsilon^2}{3\sqrt{2n}}, \frac{n\epsilon^2}{60},\frac{n\epsilon\sqrt{\epsilon}}{24\sqrt{p-1}}
\big\}=\frac{n\epsilon\sqrt{\epsilon}}{24\sqrt{p-1}}$$ such that
\begin{equation}
 6\sigma \big(2(p-1)\eta+5\big) \leq \frac n2 \epsilon^2+\frac n2 \epsilon^2=n \epsilon^2.
\end{equation}
Then (\ref{11-ineq1}) reduces to
\begin{eqnarray*}
\frac{\partial}{\partial t}\int_{M_t}f_\sigma^pd\mu_t
\leq-2\sigma p\bar{K}\int_{M_t}f_\sigma^{p}d\mu_t,
\end{eqnarray*}
and this implies
\begin{eqnarray}
\int_{M_t}f_\sigma^pd\mu_t\leq
e^{-2\sigma p\bar{K}t}\int_{M_{t_0}}f_\sigma^{p}d\mu_t.
\end{eqnarray}
\end{proof}

Then we can proceed by a Stampacchia iteration procedure as in \cite{B,H3} to bound $f_\sigma$ in $L^\infty$ and complete the proof of Proposition \ref{san1}.

Applying Proposition \ref{san1}, we have the following gradient estimation.
\begin{proposition}[\cite{B} Theorem 5.8]
For every $\eta>0$, there exists a constant $C_\eta$ depending only on $\eta$ such that for all time $t$, there holds
$$|\nabla H|^2\leq (\eta |H|^4+C_\eta)e^{-\sigma t}.$$
\end{proposition}

To estimate the diameter of $M_{t}$, we need the lower bound of the Ricci curvature along the mean curvature flow. By Proposition 2 of {\cite{MR1458750}}, the Ricci curvature of $M$ satisfies
  \begin{equation*}  \operatorname{Ric}_M \geq \frac{n-1}{n} \left( n\bar K + \frac{1}{n} | H |^{2} - | \mathring{A}
     |^{2} - \frac{n-2}{\sqrt{n ( n-1 )}} | H | |\mathring{A}| \right) . \end{equation*}
  From the pinching condition $| \mathring{A} |^{2} < \mathring{a} - \epsilon \omega$ and Lemma \ref{app} (iii), we
  obtain $\operatorname{Ric}_M \geq \frac{n-1}{n}  \epsilon    \omega >\frac{\epsilon}{n}
   | H |^{2}$.

Combining the gradient estimation and the well-known Myers theorem we can get the following lemma easily.
\begin{lemma}[\cite{LX2} Lemma 6.2]
  Suppose that $M$ is an n-dimensional
  submanifold in $\mathbb{S}^{n+p}(\frac{1}{\sqrt{\bar K}})$ satisfying $|\mathring{A}|^{2} <
  \mathring{a}-\epsilon\omega$ and $| \nabla H | <2  \eta^{2}   \max_{M} |
  H |^{2}$, where $0< \eta < \epsilon$.
  Then we have
  $ \frac{\min_{M} | H |^{2}}{\max_{M} | H |^{2}} >1- \eta $
  and $ \operatorname{diam}  M \leq ( 2  \eta  \max_{M} | H |
     )^{-1}$.
\end{lemma}

Now we can complete the proof of the theorem \ref {main2} for $n\geq7$.
\begin{proof}[Proof of theorem \ref {main2}]
From Lemma 4.7, we know $\frac{\min_{M} |H|}{\max_{M} |H|}\rightarrow 1$ and the diameter of $M_t$ is bounded alone the mean curvature flow. By similar arguments as in \cite{B},  we obtain the mean curvature flow with initial value $F$
either converges to a round point in finite time, or
converges to a total geodesic sphere of $\mathbb{S}^{n+p}(\frac{1}{\sqrt{\bar K}})$ as
$t\rightarrow\infty$.
\end{proof}

\section{Another sharp convergence theorem}
Putting
\begin{equation}\label{gamma}
  b(x)=(1-\delta)\Big(\frac{x}{n-1}+2 \bar K\Big)+\delta \alpha(x), \ \ \mathring{b}(x)=b(x)-\frac{x}{n},
\end{equation}
where $\delta=\begin{cases}
            \frac{\sqrt{12n+9}-7}{2(n-2)}, \ & 4\leq n\leq12, \\
            \frac{2(2n-5)}{n^2-2}, \ & n\geq 13.
          \end{cases}$
Then we prove the following sharp convergence theorem.

\begin{thm}\label{t2}
Let $F_0 :M \rightarrow \mathbb{S}^{n+p}(\frac{1}{\sqrt{\bar K}})$ be an n-dimensional $(n\geq4)$ smooth compact submanifold immersed in the sphere. If $M$ satisfies
\begin{equation}\label{cons2}|A|^2 \leq b(|H|^2),
\end{equation}
then the mean curvature flow with the initial value $F_0$  converges to a round point in finite time, or
converges to a total geodesic sphere of $\mathbb{S}^{n+p}(\frac{1}{\sqrt{\bar K}})$ as
$t\rightarrow\infty$.
\end{thm}

\begin{lemma}
  $\label{app2}$For $x \geq 0$,
  We have the following properties.
  \begin{enumerateroman}
    \item $\frac{4x (\mathring{b}')^2}{\mathring{b}}  < 1,$

    \item $2x  \mathring{b}''  + \mathring{b}'  < \frac{2
    ( n-1 )}{n ( n+2 )}, \ \  for \ \delta\leq \frac{2(2n-5)}{n^2-4},$
    \item $\frac{n-2}{\sqrt{n(n-1)}}\sqrt{x\mathring{b}}+\mathring{b}  < \frac{x}{n}+n\bar K,$

    \item $\mathring{b}  \cdot \left(
    b -n\bar K  \right)-x\mathring{b}' \cdot \left( b +n\bar K  \right) < -2(1-\delta)(n-2)\bar K^2 $,

    \item $\frac{x}{n-1}\left( b  +n\bar K
    \right) - (\frac{x}{n-1}+2n\bar K) \left(  \mathring{b} +
    b  -n\bar K - x\mathring{b}' \right) \\
     < -\frac{2\bar Kx}{n-1}+ 2n(n-4)\bar K^2 $,

    \item $2\mathring{b}-\frac{x}{n}+x\mathring{b}'< 2\big(\delta (n-2)+2\big)\bar K$.
  \end{enumerateroman}
\end{lemma}

\begin{proof}
  By direct computations, we get
  \[ \mathring{b}=\delta n\bar K+2(1-\delta)\bar K+\frac{\delta n^{2} -2\delta n+2}{2 ( n-1 ) n}x-\frac{\delta(n-2)}{2(n-1)}\sqrt{x^{2} +4 ( n-1 ) \bar K x},  \]

  \[ \mathring{b}'  = \frac{\delta n^{2} -2\delta n+2}{2 ( n-1 ) n} - \frac{\delta(n-2)}{2
     ( n-1 )}   \frac{x+2 ( n-1 ) \bar K}{\sqrt{x^{2} +4 ( n-1 ) \bar K x}}\leq\frac{n}{n(n-1)} ,  \]

  \[\mathring{b}''  = \frac{2 \delta ( n-1 ) ( n-2 ) \bar K^{2}}{\big(
     x^{2} +4 ( n-1 ) \bar K x \big)^{3/2}} .\]

(i)
\begin{equation*}
  \frac{4x (\mathring{b}')^2}{\mathring{b}}  < \frac{x }{n(n-1)\mathring{b}}<1.
\end{equation*}

 (ii)
 \begin{align*}
   2x  \mathring{b}''  + \mathring{b}'  =& \frac{\delta n^{2} -2\delta n+2}{2 ( n-1 ) n} - \frac{\delta(n-2)\big(x^3+3 ( n-1 ) \bar Kx^2 \big)}{2(n-1)\big(
     x^{2} +4 ( n-1 ) \bar K x \big)^{3/2}} \\
     < &\frac{\delta n^{2} -2\delta n+2}{2 ( n-1 ) n}\leq \frac{2
    ( n-1 )}{n ( n+2 )}, \ \ \ \ \rm{where}\  \delta\leq \frac{2(2n-5)}{n^2-4}.
 \end{align*}

(iii)
\begin{align*}
   & \frac{n-2}{\sqrt{n(n-1)}}\sqrt{x\mathring{b}}+\mathring{b} < \frac{n-2}{\sqrt{n(n-1)}}\sqrt{x\mathring{\alpha}}+\mathring{\alpha}=\frac{x}{n}+n\bar K.
\end{align*}

 (iv)
 \begin{align*}
   &a( \mathring{b}-x\mathring{b}') \\
   =& a \left( \delta n\bar K+2(1-\delta)\bar K-\frac{\delta(n-2)\bar K x}{\sqrt{x^{2} +4 ( n-1 ) \bar K x}} \right) \\
    =& \big(\delta n\bar K+2(1-\delta)\bar K\big)^2-\big(\delta n\bar K+2(1-\delta)\bar K\big)\cdot \frac{\delta(n-2)\big(x^2+3 ( n-1 ) \bar Kx \big)}{(n-1)\sqrt{x^{2} +4 ( n-1 ) \bar K x}}\\
    &+\frac{\delta^2(n-2)^2+\big(\delta n+2(1-\delta)\big)^2}{2(n-1)}\bar Kx.
 \end{align*}

 \begin{align*}
   &n\bar K(\mathring{b}+x\mathring{b}') \\
   =& n\bar K \left(  \delta n\bar K+2(1-\delta)\bar K+\frac{\delta n^{2} -2\delta n+2}{ ( n-1 ) n}x-\frac{\delta(n-2)\big(x^2+3 ( n-1 ) \bar Kx \big)}{(n-1)\sqrt{x^{2} +4 ( n-1 )\bar Kx}} \right).
 \end{align*}

 \begin{align*}
   &a( \mathring{b}-x\mathring{b}')-n\bar K(\mathring{b}+x\mathring{b}') \\
   =& (\delta-1)(n-2) \left( \delta n\bar K+2(1-\delta)\bar K- \frac{\delta(n-2)\big(x^2+3 ( n-1 ) \bar Kx \big)}{(n-1)\sqrt{x^{2} +4 ( n-1 )\bar Kx}} \right)\bar K\\
   &+\frac{\delta(\delta-1)(n-2)^2}{n-1}\bar Kx\\
   =&(\delta-1)(n-2)\bar K \left( \delta n\bar K+2(1-\delta)\bar K+\frac{\delta(n-2)}{n-1}\bigg(x-\frac{x^2+3 ( n-1 ) \bar Kx }{\sqrt{x^{2} +4 ( n-1 )\bar Kx}}\bigg)\right)\\
   <&-(1-\delta)(n-2)\bar K \Bigg( \delta n\bar K+2(1-\delta)\bar K+\frac{\delta(n-2)}{n-1} \bigg( -(n-1)\bar K \bigg) \Bigg)\\
   =&-2(1-\delta)(n-2)\bar K^2.
 \end{align*}

 (v)
 \begin{align*}
   &  \frac{x}{n-1}\left( b   +n\bar K
    \right) - (\frac{x}{n-1}+2n\bar K) \left(  b
      -n\bar K+\mathring{b} - x\mathring{b}' \right) \\
    =& \frac{2nx\bar K}{n-1}-2n\bar K(a-n\bar K)\\
    &-(\frac{x}{n-1}+2n\bar K) \left(  \delta n\bar K+2(1-\delta)\bar K-\frac{\delta(n-2)\bar K x}{\sqrt{x^{2} +4 ( n-1 ) \bar K x}} \right)\\
    =&-\frac{2\bar Kx}{n-1}+ 2n(n-4)\bar K^2+\frac{\delta(n-2)\bar Kx}{n-1} \left(\frac{x}{\sqrt{x^{2} +4 ( n-1 ) \bar K x}}-1\right)\\
    &+\frac{\delta n(n-2)\bar K}{n-1} \left(\frac{x^{2} +6 ( n-1 ) \bar K x}{\sqrt{x^{2} +4 ( n-1 ) \bar K x}}-x-4(n-1)\bar K\right)\\
    <&-\frac{2\bar Kx}{n-1}+ 2n(n-4)\bar K^2.
 \end{align*}

(vi)
\begin{align*}
   &  2\mathring{b}-\frac{x}{n}+x\mathring{b}'\\
    =&2\big(\delta (n-2)+2\big)\bar K+3\frac{\delta n^{2} -2\delta n+2}{2 ( n-1 ) n}x-\frac{x}{n}\\
    &-2\frac{\delta(n-2)}{2(n-1)}\sqrt{x^{2} +4 ( n-1 ) \bar K x}- \frac{\delta(n-2)}{2
     ( n-1 )}   \frac{x^2+2 ( n-1 ) \bar Kx}{\sqrt{x^{2} +4 ( n-1 ) \bar K x}}\\
     =&2\big(\delta (n-2)+2\big)\bar K-\frac{n-4}{ ( n-1 ) n}x\\
    &+\frac{3\delta (n-2)x}{2 ( n-1 ) }- \frac{\delta(n-2)}{2
     ( n-1 )}   \frac{3x^2+10 ( n-1 ) \bar Kx}{\sqrt{x^{2} +4 ( n-1 ) \bar K x}}\\
     \leq&2\big(\delta (n-2)+2\big)\bar K.
 \end{align*}

\end{proof}

There exists a small positive number $0<\epsilon\ll\frac{1}{n^2},$
such that $M_{0}$ satisfies
\begin{equation}\label{con-w2}
  |\mathring{A}|^{2} < \mathring{b} - \epsilon   \omega , \hspace{2em}
   \omega = \frac{| H |^{2}}{n-1} +2n\bar K.
\end{equation}
In the following we prove that the pinching condition above is
preserved along the flow.

\begin{proposition}
Let $F_0 :M \rightarrow \mathbb{S}^{n+p}(\frac{1}{\sqrt{\bar K}})$ be a compact submanifold immersed in the sphere. Suppose there exists a small positive number $\epsilon(\ll \frac{1}{n^2})$ such that $ |A| ^ 2 \leq b(|H|^2)-\epsilon\omega, $
 then this condition holds
  along the mean curvature flow for all time $t \in [ 0,T )$ where $T\leq\infty$.
\end{proposition}
\begin{proof}
As the proof of Proposition \ref{baochi}, we need the following inequality holds for $x\geq0$.
\begin{align*}
b( \mathring{b}-x\mathring{b}')-n\bar K(\mathring{b}+x\mathring{b}')
       +P_{2} \left( 2 \mathring{b} - \frac{x}{n}  + x \mathring{b}' \right)- \frac{3}{2} P^2_{2}<0. \\
\end{align*}

From Lemma \ref{app2}, we have
  \begin{align*}
  &b( \mathring{b}-x\mathring{b}')-n\bar K(\mathring{b}+x\mathring{b}')
       +P_{2} \left( 2 \mathring{b} - \frac{x}{n}  + x \mathring{b}' \right)- \frac{3}{2} P^2_{2}\\
     <&  -2(1-\delta)(n-2)\bar K^2+2\Big( \delta (n-2)+2 \Big)\bar K P_2-\frac{3}{2}P^2_2.
  \end{align*}
Its discriminant is
 \begin{align*}
    \Delta =& \Big( \delta (n-2)+2 \Big)^2-3(1-\delta)(n-2) \\
     =& \delta^2 (n-2)^2 +7\delta (n-2)+4-3(n-2)\\
     \leq& 0.
  \end{align*}
\end{proof}

We complete the proof of Theorem \ref{t2} with the following estimate.
\begin{lemma}\label{}Let $F :M \rightarrow \mathbb{S}^{n+p}(\frac{1}{\sqrt{\bar K}})$ be a compact submanifold immersed in the sphere with
constant curvature $\bar K$. If $F$ satisfies pinching condition (\ref{con-w2}), then there exists a strictly positive constant $\epsilon$ such that
\begin{equation}
  n\bar K|\mathring A|^2-R_1+R_3\geq
           \frac{n}{2}|\mathring A|^2(\epsilon|A|^2-4\bar{K}).
\end{equation}
\end{lemma}
\begin{proof}
\begin{equation*}
  b\leq\frac{|H|^2}{n-1}+(1-\delta)2\bar K+\delta n \bar K.
\end{equation*}
From Lemma \ref{eq}, we have
\begin{align*}
   n\bar K|\mathring A|^2-R_1+R_3&\geq\frac{n}{2}|\mathring A|^2(\frac{|H|^2}{n(n-1)}+2\bar K)-\frac n2|\mathring A|^4\\
   &=\frac{n}{2}|\mathring A|^2(\frac{|H|^2}{n-1}+2\bar K-|A|^2)\\
   &\geq\frac{n}{2}|\mathring A|^2(\frac{|H|^2}{n-1}+2\bar K-b+\epsilon\omega)\\
   &\geq\frac{n}{2}|\mathring A|^2\Big(\epsilon\omega-\delta(n-2)\bar K\Big).
\end{align*}

\end{proof}

$$ \ \ \ \ $$

\end{document}